\documentclass[11pt, reqno, a4paper]{amsart}

\usepackage{amsmath, amssymb, amsthm}
\usepackage{amscd}
\usepackage[figuresright]{rotating}

\usepackage{color}

\usepackage{etex}
\usepackage{dsfont}
\usepackage[matrix, arrow, curve, frame, arc]{xy}
\usepackage{pgf}
\usepackage{tikz-cd}
\usetikzlibrary{arrows,shapes,automata,backgrounds,petri,calc}
\usepackage[square,sort,comma,numbers]{natbib}
 \usepackage{url}
 \usepackage{hyperref}
 \usepackage[shortlabels]{enumitem} 

\newcommand{\tikzAngleOfLine}{\tikz@AngleOfLine}
\def\tikz@AngleOfLine(#1)(#2)#3{%
\pgfmathanglebetweenpoints{%
\pgfpointanchor{#1}{center}}{%
\pgfpointanchor{#2}{center}}
\pgfmathsetmacro{#3}{\pgfmathresult}%
}

\newcommand{\End}{\operatorname{End}}
\newcommand{\Hom}{\operatorname{Hom}}
\newcommand{\gldim}{\operatorname{gldim}}

\newcommand{\Ext}{\operatorname{Ext}}
\newcommand{\Tor}{\operatorname{Tor}}
\newcommand{\add}{\!\operatorname{add}}
\newcommand{\pdim}{\operatorname{pdim}}

\newcommand{\m}{\!\operatorname{-mod}} 
 
\newcommand{\rmod}{\operatorname{mod-}\!\!}

\newcommand{\rad}{\operatorname{rad}}

\renewcommand{\top}{\operatorname{top}}

\newcommand{\im}{\!\operatorname{im}} 

\newcommand{\coker}{\operatorname{coker}}

\newcommand{\soc}{\operatorname{soc}}

\newcommand{\domdim}{\operatorname{domdim}}

\newcommand{\Tr}{\operatorname{Tr}}

\newcommand{\grade}{\operatorname{grade}}
\newcommand{\Dom}{\operatorname{Dom}}
\newcommand{\TF}{\operatorname{TF}}

\newtheorem{numberingthm}{Theorem}[section] 
\theoremstyle{definition}

\newtheorem{example}[numberingthm]{Example}

\theoremstyle{plain}
\newtheorem{Prop}[numberingthm]{Proposition}
\newtheorem{Theorem}[numberingthm]{Theorem}
\newtheorem{Conjecture}[numberingthm]{Conjecture}
\newtheorem{Cor}[numberingthm]{Corollary}
\newtheorem{Lemma}[numberingthm]{Lemma}
\newtheorem{Question}[numberingthm]{Question}

\newtheorem*{nonumberthm}{Theorem}

\newtheoremstyle{cited}%
  {4pt}
  {4pt}
  {\itshape}
  {}
  {\bfseries}
  {.}
  {0.5em}
  {\thmname{#1}  \thmnote{\normalfont#3}}

\theoremstyle{cited}
\newtheorem*{thmintroduction}{Theorem}
\newtheorem*{corintroduction}{Corollary}

\theoremstyle{remark}

\newcommand{\Thmref}{Theorem~}

\newcommand{\Corref}{Corollary~}

\begin{document}

\title[]{Higher torsion-free Auslander-Reiten sequences and the dominant dimension of algebras}

\author[T. Cruz]{Tiago Cruz}
\address[Tiago Cruz]{Institut f\"ur Algebra und Zahlentheorie,
Universit\"at Stuttgart, Germany }
\email{tiago.cruz@mathematik.uni-stuttgart.de}

\author[R. Marczinzik]{Ren\'{e} Marczinzik}
\address[Ren\'{e} Marczinzik]{Mathematical Institute of the University of Bonn, Endenicher Allee 60, 53115 Bonn, Germany}
\email{marczire@math.uni-bonn.de}
\date{\today}

\begin{abstract}
	We generalise a theorem of Tachikawa about reflexive Auslander-Reiten sequences. We apply this to give a new characterisation of the dominant dimension of gendo-symmetric algebras. We also generalise a formula due to Reiten about the dominant dimension of an algebra $A$ and grades of torsion $A$-modules.
\end{abstract}
\subjclass[2020]{Primary: 16G10, Secondary: 16E10}
\keywords{reflexive modules, Auslander-Reiten theory, dominant dimension, grade}

\maketitle

\section{Introduction}
The concepts of grade, torsion freeness and torsion play an important role in commutative algebra, whereas Auslander--Reiten theory provides a toolset to visualise and study the representation theory of a finite-dimensional algebra. The main goal of this paper is to combine these concepts using homological techniques. This effort leads us to new characterisations of dominant dimension.
Given a noetherian ring $A$, 
let $(-)^{*}$ denote $\Hom_A(-,A)$.
To every $A$-module $M$ we can associate the canonical evaluation map $f_M : M \rightarrow M^{**}$, that sends $m$ to the map $\psi_m$ with $\psi_m(f)=f(m)$ for $f \in M^{*}$. Recall that $M$ is called \emph{torsionless} if $f_M$ is injective and \emph{reflexive} if $f_M$ is an isomorphism.
The study of reflexive modules is a classical topic in commutative algebra and algebraic geometry, we refer for example to \cite{En} and \cite{EK}. In the following assume that $A$ is a finite dimensional algebra over a field $k$.
For a natural number $n \geq 1$, a module $M$ is called \emph{$n$-torsion-free} if $\Ext_A^i(D(A), \tau(M))=0$ for all $i$ with $1 \leq i \leq n$, where $\tau$ denotes the Auslander-Reiten translate.
It is well known that a module is torsionless if and only if it is 1-torsion-free and a module is reflexive if and only if it is 2-torsion-free, see for example \citep[IV, Corollary 3.3]{ARS}.
Thus, $n$-torsion-free modules can be seen as a generalisation of torsion-free and reflexive modules. Easy examples of $n$-torsion-free modules for all $n \geq 1$ are projective modules.
In \cite{Ta}, Tachikawa defined an almost split sequence of the form $0 \rightarrow U \rightarrow X \rightarrow \tau^{-1}(U)\rightarrow 0$ to be reflexive if every module in the short exact sequence is reflexive.
He then defined for an algebra $A$ to have \emph{reflexive Auslander-Reiten sequences} if every almost split sequence of the form $0 \rightarrow U \rightarrow X \rightarrow \tau^{-1}(U) \rightarrow 0$ is reflexive for every indecomposable projective non-injective module $U$.
He characterised algebras with reflexive Auslander-Reiten sequences as follows, see \citep[Theorem 2.4.]{Ta}:  
\begin{nonumberthm}
	Let $A$ be a finite-dimensional algebra. Then $A$ has reflexive Auslander-Reiten sequences if and only if the following two conditions are satisfied:
	\begin{enumerate}
		\item $\domdim(A) \geq 2$
		\item $\domdim(R) \geq 4$, when $R:=\End_A(A \oplus D(A))$.
	\end{enumerate}
\end{nonumberthm}

In this article, we generalise Tachikawa's theorem to a higher version using $n$-torsion-free modules.
Namely, we define an almost split sequence of the form $0 \rightarrow U \rightarrow X \rightarrow \tau^{-1}(U) \rightarrow 0$  to be $n$-torsion-free if every module in the short exact sequence is $n$-torsion-free. An algebra is then said to have \emph{$n$-torsion-free Auslander-Reiten sequences} for a natural number $n \geq 1$ if every almost split sequence of the form $0 \rightarrow U \rightarrow X \rightarrow \tau^{-1}(U) \rightarrow 0$  is $n$-torsion-free for all indecomposable projective non-injective modules $U$.
Our main result is as follows:
\begin{thmintroduction}[\bfseries \ref{mainresult}]
	Let $A$ be a finite-dimensional algebra and $n \geq 1$ a natural number. Then $A$ has $n$-torsion-free Auslander-Reiten sequences if and only if the following two conditions are satisfied:
	\begin{enumerate}
		\item $\domdim(A) \geq n$
		\item $\domdim(R) \geq n+2$, when $R:=\End_A(A \oplus D(A))$.
	\end{enumerate}
\end{thmintroduction}
We give two applications of this generalisation of Tachikawa's result. The first application gives a new characterisation of the dominant dimension of gendo-symmetric algebras.
In \cite{FanKoe2}, Fang and Koenig defined \emph{gendo-symmetric algebras} as algebras of the form $\End_A(M)$, where $A$ is a symmetric algebra and $M$ a generator of $\rmod A$.
Gendo-symmetric algebras clearly generalise the important class of symmetric algebras but also contain many non-symmetric algebras such as Schur algebras $S(n,r)$ for $n \geq r$ and blocks of category $\mathcal{O}$, see for example \cite{KSX}. 
We apply our main theorem to give a new characterisation of the dominant dimension for gendo-symmetric algebras:
\begin{corintroduction}[\bfseries\ref{cor3dot3}] 
	Let $A$ be a gendo-symmetric algebra.
	Then $A$ has $n$-torsion-free Auslander-Reiten sequences if and only if it has dominant dimension at least $n+2$.
\end{corintroduction}
Recall that an algebra $A$ is called an \emph{higher Auslander algebra} if $\gldim A \leq n \leq \domdim A$ for some $n \geq 2$.
Iyama's celebrated higher Auslander correspondence, see \cite{Iya}, showed that higher Auslander algebras are in bijective correspondence with cluster tilting modules.
Among algebras with cluster tilting modules, the most important class is the $d$-representation-finite algebras that can be seen as a generalisation of the representation-finite hereditary algebras, which are exactly the 1-representation-finite algebras.
An algebra $A$ is \emph{$d$-representation-finite} if it has global dimension at most $d$ and a $d$-cluster-tilting object.
A second application of our main result gives:
\begin{corintroduction}[\bfseries\ref{cor3dot7}]
	Let $B$ be a finite-dimensional algebra with a $d$-cluster tilting object $M$.
	Then the higher Auslander algebra $A=\End_B(M)$ has $d$-torsion-free Auslander-Reiten sequences if and only if $B$ is $d$-representation-finite.
\end{corintroduction}
The modules $t(M):=\ker f_M$ which appear as kernels of the maps $f_M:M \rightarrow M^{**}$ for some module $M$ are called \emph{torsion modules}.
It is natural to try to use torsion modules to understand dominant dimension since torsion modules can be regarded as the "$0$-torsion-free" modules.
Our last main result generalises a theorem due to Reiten (\cite[Theorem 5.2]{Rei}) for finite-dimensional algebras, which gives a new formula for the dominant dimension using torsion modules:
\begin{thmintroduction}[\bfseries\ref{prop6dot13}]
	Let $A$ be a finite-dimensional algebra over a field. Assume that $A$ has dominant dimension at least one. Then,
	\begin{align*}
		\domdim A &=	\inf\{\grade(\Ext_A^i(M, A))\colon M\in A\m, i\geq 1 \} = \inf\{\grade(t(M))\colon M\in A\m \} \\&= \inf\{\grade(t(S))\colon S \text{ simple in } A \},
	\end{align*}
 where $t(M)=\ker f_M.$
\end{thmintroduction}
In \Thmref \ref{thm4dot6}, we show that these equalities also hold if we replace the assumption of $A$ having positive dominant dimension by 
requiring that all torsion modules have grade greater than one. As an application of \Thmref \ref{prop6dot13}, we describe in \Corref\ref{cor4dot7} the projective dimension of all torsion modules over higher Auslander algebras.

\section{Preliminaries}

We will start by presenting the notation and further concepts to be used throughout the paper. We refer for example to \cite{ARS} for an introduction to the representation theory of finite-dimensional algebras and more information on the concepts that we use here. Let $A$ be a finite-dimensional algebra over a field $k$. Throughout this paper, we assume that all modules are finitely generated, and all algebras are finite-dimensional algebras over a field unless stated otherwise. We write $A\m$ to denote the category of finitely generated left $A$-modules, whereas $\rmod A$ denotes the category of finitely generated right $A$-modules. By $D$ we mean the \emph{standard duality} $\Hom_k(-, k)\colon \rmod A \longleftrightarrow A\m$.
Given a morphism of $A$-modules $f$, we write $\im f$ to denote the image of $f$ and $\ker f$ denotes the kernel of $f$. We say that an $A$-module is a \emph{generator} if it contains all projective indecomposable $A$-modules as direct summands. Dually, the concept of \emph{cogenerator} is defined.

Given an $A$-module $M$, $\Tr M$ denotes the \emph{transpose of $M$}, while $\Omega^iM$ denotes the $i$-th syzygy of $M$ or the $(-i)$-th cosyzygy of $M$ if $i$ is a negative integer.
We denote by $\tau$ the \emph{Auslander-Reiten translation} $D\Tr$, whereas $\tau^{-1}$ denotes the \emph{inverse Auslander-Reiten translation} $\Tr D$. 
We write $\Omega^m(\rmod A)$ to denote the collection of all modules $M$ over $A$ that can be written as $M\simeq \Omega^m(X)$ for some $X\in \rmod A$ together with all projective $A$-modules. The collection of all $m$-torsion-free $A$-modules is denoted by $\TF_m$.  

The \emph{grade} of a non-zero $A$-module $M$, $\grade(M)$, is defined as $\inf \{ n \geq 0 \mid \Ext_A^n(M, A)\neq 0 \}$, where $\Ext_A^0(M, A)=\Hom_A(M,A)$.

A detailed exposition on the interconnections between grade, $n$-torsion-free modules and syzygies can be found in \cite{AB69}.

As it will be illustrated below, it makes sense to study these concepts together with the homological invariant known as the dominant dimension.
Given an $A$-module $M$, let $0\rightarrow M\rightarrow I_0\rightarrow I_1\rightarrow \cdots $ be a minimal injective coresolution of $M$. The module $M$ is said to have \emph{dominant dimension} $n$, denoted as $\domdim_A M$, if $I_0, \ldots, I_{n-1}$ are projective and $I_n$ is not. It has infinite dominant dimension if all terms $I_i$ are projective. We write $\domdim A=\domdim_A A$ and we denote by $\Dom_m$ the collection of all right $A$-modules having dominant dimension at least $m$, for a natural number $m$. This homological invariant is an effective tool to obtain ring theoretical properties of the underlying algebra and to simplify the study of its representation theory. In particular, all finite-dimensional algebras of dominant dimension at least two are precisely the endomorphism algebras of generator-cogenerators (see for example \cite{zbMATH03248955}). 

Below we list three results about the dominant dimension that are crucial for this work.

\begin{Theorem} \label{connectiontorsionfreedomdim}
	Let $A$ be an algebra of dominant dimension at least $n \geq 1$. 
	Then we have:
	$$\Dom_m= \Omega^m(\rmod A)=\TF_m$$ for all $1\leq m \leq n.$
\end{Theorem}
\begin{proof}
The equality $\Dom_m= \Omega^m(\rmod A)$ can be found in \cite[Proposition 4]{MV} and the equality $\Omega^m(\rmod A)=\TF_m$ is a consequence of \citep[Theorem 0.1, Proposition 1.6]{AR}.
\end{proof}

The next theorem is often called Mueller's theorem.
\begin{Theorem} \label{muellertheo}
	Let $A$ be an algebra and $M$ a generator-cogenerator of $\rmod A$. Let $B:=\End_A(M)$. Then the dominant dimension of $B$ is equal to $\inf \{ i \geq 1 | \Ext_A^i(M,M) \neq 0 \} +1$.
\end{Theorem}
\begin{proof}
	See \citep[Lemma 3]{zbMATH03248955}.
\end{proof}

\begin{Theorem} \label{fankoeresult}
	Let $A$ be a gendo-symmetric algebra. Then, $$\domdim A=\inf \{i \geq 1 | \Ext_A^i(D(A),A) \neq 0 \} +1.$$
	
\end{Theorem}
\begin{proof}
	See \cite[Proposition 3.3]{FanKoe2}.
\end{proof}

For further background material on the dominant dimension we refer for example to \cite{Yam}. We recall the following well-known result that allows us to determine extensions of right modules via extension of left modules and vice-versa.

\begin{Lemma}
    Let $A$ be a finite-dimensional algebra over a field. Then, for every $i\geq 0$ and $M, N\in A\m$ we have 
    $\Ext_A^i(M, N)\simeq \Ext_A^i(DN, DM)$.
\end{Lemma}

\section{Algebras with $n$-torsion-free Auslander-Reiten sequences}

Our aim in this section is to characterise the existence of torsion-free Auslander--Reiten sequences in terms of dominant dimension.

\begin{Theorem} \label{nTFimpliesdomdim}
	Let $A$ have $n$-torsion-free Auslander-Reiten sequences for some $n\geq 1$. Then $\domdim A\geq n$.
\end{Theorem}
\begin{proof}
	Assume that $eA$ is not injective for some idempotent $e\in A$. Recall that the socle of the injective module $D(eA)$ is $\top Ae$.
 We aim to show that for $i=0, \ldots, n-1$ we have $\Ext_A^i(\top Ae, A)=0$.
 
	Observe that $\Ext_A^i(\top Ae, A)=\Hom_A(\top Ae, E_i)$, where $0\rightarrow A\rightarrow E_0\rightarrow E_1\rightarrow \cdots$ is a minimal injective coresolution of $A$ (see for instance \citep[Lemma 1]{zbMATH03620079}). 
	Indeed, the maps $\Hom_A(S, E_i)\rightarrow \Hom_A(S, \Omega^{-(i+1)}(A))$ are always zero whenever $S$ is simple. Fix $S$ a simple module. Pick a non-zero map $f\in \Hom_A(S, E_i)$ and denote by $g$ the surjective map ${E_i\rightarrow \Omega^{-(i+1)}(A)}$. By the minimality, $f(S)\cap \Omega^{-i}(A)\neq 0$ (otherwise $f(S)=0$ since $\Omega^{-i}(A)\rightarrow E_i$ is an injective hull). Thus $\im f$ is contained in $\Omega^{-i}(A)$ because $S$ is simple. Thus, $g\circ f=0$.
	
	Hence, $\Ext_A^i(\top Ae, A)\simeq \Hom_A(\top Ae, E_i)$ for all $i\geq 0$. 
	
	Let \begin{align}
		0\rightarrow eA\rightarrow Y\rightarrow Z\rightarrow 0\label{eq1}
	\end{align}be an almost split sequence. By assumption, it is $n$-torsion-free. 
	Observe that for any module $X$ which is $n$-torsion-free we obtain that 
	\begin{align}
		0=\Ext_A^i(DA, \tau X)=\Ext_A^i(\Tr X, DDA)=\Ext_A^i(\Tr X, A), \quad i=1, \ldots, n.
	\end{align}
	Moreover, consider the minimal projective presentation for $X$, $P_1\rightarrow P_0\rightarrow X\rightarrow 0$. Applying $(-)^*:=\Hom_A(-, A)$ yields the exact sequence
	\begin{align}
		0\rightarrow X^*\rightarrow P_0^*\rightarrow P_1^*\rightarrow \Tr X\rightarrow 0. \label{eq3}
	\end{align}Hence applying $\Hom_A(-, A)$ yields $\Ext_A^i(X^*, A)\simeq \Ext_A^{i+2}(\Tr X, A)$ for all $i>0$. Thus, $\Ext_A^i(X^*, A)=0$ for $1\leq i\leq n-2$ and every $n$-torsion-free module $X$. Therefore, \begin{align}
		\Ext_A^i(Y^*, A)=\Ext_A^i(Z^*, A)=0, \quad 1\leq i\leq n-2. \label{eq4}
	\end{align}
	
	Applying $\Hom_A(-, A)$ to (\ref{eq1}) we obtain the exact sequence $$0\rightarrow \Hom_A(Z, A)\rightarrow \Hom_A(Y, A)\rightarrow \Hom_A(eA, A).$$ Observe that $\im(\Hom_A(Y, A)\rightarrow \Hom_A(eA, A))\subset \Hom_A(eA, A)\simeq Ae$. Since (\ref{eq1}) is almost split, every map $eA\rightarrow wA$ factors through $Y$ for every primitive idempotent $w\neq e$. Given $f\in \Hom_A(eA, A)$ so that $f(e)\in \rad Ae$, the morphism $eA\xrightarrow{f}A \twoheadrightarrow eA$ is not an isomorphism, and therefore it factors through $Y$. In such case, all the maps $eA \xrightarrow{f}A\twoheadrightarrow wA$ factors through $Y$ for all primitive idempotents $w$, and so $f$ factors through $Y$. This shows that $\rad Ae\subset  \im(\Hom_A(Y, A)\rightarrow \Hom_A(eA, A))$. Since $\rad(Ae)$ is the maximal submodule of $Ae$, this must be an equality, otherwise (\ref{eq1}) would be a split exact sequence.
	We obtained an exact sequence
	\begin{align}
		0\rightarrow \Hom_A(Z, A)\rightarrow \Hom_A(Y, A)\rightarrow \rad Ae\rightarrow 0,\label{eq2}
	\end{align}as left $A$-modules and the map $Y^*\rightarrow (eA)^*\simeq Ae$ factors through $\rad Ae$. Consider the exact sequence $0\rightarrow \rad Ae\rightarrow Ae\rightarrow \top Ae\rightarrow 0$.  Applying $(-)^*:=\Hom_A(-, A)$ yields the exact sequence
	\begin{align}
		0\rightarrow (\top Ae)^*\rightarrow (Ae)^*\rightarrow (\rad Ae)^*\rightarrow \Ext_A^1(\top Ae, A)\rightarrow 0\label{eq5}
	\end{align} and $\Ext_A^{i+1}(\top Ae, A)\simeq \Ext_A^i(\rad Ae, A)$ for all $i>0$. Further, the map $(eA)^{**}\simeq (Ae)^*\rightarrow Y^{**}$ factors through $(\rad Ae)^*$. Hence,
	applying $\Hom_A(-, A)$ to (\ref{eq2}) yields the following commutative diagram with exact rows
	\begin{equation}
		\begin{tikzcd}
			0 \arrow[r] &eA \arrow[r] \arrow[d, "\simeq"] & Y \arrow[r] \arrow[dd]& Z \arrow[r] \arrow[dd]& 0 \\
			& (Ae)^* \arrow[d]& & \\
			0 \arrow[r] & (\rad Ae)^* \arrow[r] & Y^{**} \arrow[r] & Z^{**}\arrow[r] &\Ext_A^1(\rad Ae, A)
		\end{tikzcd}\label{eq6}
	\end{equation}
	If $n\geq 1$, then the map $Y\rightarrow Y^{**}$ is injective. In such a case, the commutativity of (\ref{eq6}) yields that the map ${(Ae)^*\rightarrow (\rad Ae)^*}$ is also injective  and thus by (\ref{eq5}),  $\Hom_A(\top Ae, A)=0$. 
	If $n\geq 2$, then both maps $Y\rightarrow Y^{**}$ and $Z\rightarrow Z^{**}$ are bijective. By the Snake Lemma, $(Ae)^*\rightarrow (\rad Ae)^*$ is also bijective and the map $Y^{**}\rightarrow Z^{**}$ is therefore surjective. By (\ref{eq5}), $\Ext_A^1(\top Ae, A)=0$. Further, the surjectivity of $Y^{**}\rightarrow Z^{**}$ and applying $\Hom_A(-, A)$ to (\ref{eq2}) yields the long exact sequence
\begin{equation}
    \begin{tikzcd}
  0 \rar & \Ext_A^1(\rad Ae, A) \rar
             \ar[draw=none]{d}[name=X, anchor=center]{}
    & \Ext_A^1(Y^*, A) \rar & \Ext_A^1(Z^*, A) \ar[rounded corners,
            to path={ -- ([xshift=2ex]\tikztostart.east)
                      |- (X.center) \tikztonodes
                      -| ([xshift=-2ex]\tikztotarget.west)
                      -- (\tikztotarget)}]{dll}[at end]{} \\      
 & \Ext_A^2(\rad Ae, A) \rar & \Ext_A^2(Y^*, A) \rar & \cdots
\end{tikzcd}. \label{eq8}
\end{equation}

	By (\ref{eq4}) and (\ref{eq8}), we obtain that $\Ext_A^i(\rad Ae, A)=0$ for $1\leq i\leq n-2$. Hence, we obtain also that  $\Ext_A^{i}(\top Ae, A)=0$ for $i=0, \ldots, n-1$.
	
	Therefore, $\top Ae$ is not in the socle of $E_i$ for $0\leq i\leq n-1$ and so all $E_i$ are projective-injective for $0\leq i\leq n-1$. This means that $\domdim A\geq n$.
\end{proof}

\begin{Theorem} \label{mainresult}
	Let $A$ be a finite-dimensional algebra and $n \geq 1$ a natural number. Then $A$ has $n$-torsion-free Auslander-Reiten sequences if and only if the following two conditions are satisfied:
	\begin{enumerate}
		\item $\domdim(A) \geq n$
		\item $\domdim(R) \geq n+2$, when $R:=\End_A(A \oplus D(A))$.
	\end{enumerate}
\end{Theorem}
\begin{proof}
	Assume first that $A$ is an algebra of dominant dimension at least $n$ and such that $R=\End_A(A \oplus D(A))$ has dominant dimension at least $n+2$.
	In case $A$ is selfinjective, every module is $n$-torsion-free for any $n \geq 1$ and there is nothing to show. Thus assume that $A$ is not selfinjective and thus there exists at least one indecomposable projective non-injective $A$-module.
	Let $P$ be an indecomposable projective non-injective $A$-module and 
	$$0 \rightarrow P \rightarrow X \rightarrow \tau^{-1}(P) \rightarrow 0$$
	the associated almost split sequence. 
	To see that $\tau^{-1}(P)$ is $n$-torsion-free just note that for all $i$ with $1 \leq i \leq n$ we have $\Ext_A^i(D(A),\tau(\tau^{-1}(P))= \Ext_A^i(D(A),P)=0$ since this module is a direct summand of $\Ext_A^i(D(A),A)=0$ by assumption.
	Since $A$ has dominant dimension at least $n$, we have $\Dom_n=\TF_n$ by \Thmref\ref{connectiontorsionfreedomdim}. 
	Thus the module $\tau^{-1}(P)$ also has dominant dimension at least $n$. As a projective module $P$ also has dominant dimension at least $n$ since we assume that $A$ has dominant dimension at least $n$.
	Now by the Horseshoe lemma, the subcategory $\Dom_n$ is closed under extensions. Thus with $P$ and $\tau^{-1}(P)$, also the middle term $X$ in the almost split sequence has dominant dimension at least $n$. Using again that $\Dom_n=\TF_n$, we conclude that $X$ is also $n$-torsion-free and thus $A$ has $n$-torsion-free Auslander-Reiten sequences. \newline
	Now assume that $A$ has $n$-torsion-free Auslander-Reiten sequences. Again, the result is clear in case $A$ is selfinjective, since then $A$ has infinite dominant dimension and $R$ is Morita-equivalent to $A$ and thus also has infinite dominant dimension. Assume now that $A$ is not selfinjective and let $P$ be an indecomposable projective non-injective module. By assumption, $\tau^{-1}(P)$ is $n$-torsion-free, which is equivalent to $\Ext_A^i(D(A),P)=\Ext_A^i(D(A),\tau(\tau^{-1}(P))=0$ for all $i=1,...,n$. This gives that also $\Ext_A^i(D(A),A)=\bigoplus\limits_{P}^{}{\Ext_A^i(D(A),P)}=0$ for all $i=1,...,n$, where the direct sum is over all indecomposable projective $A$-modules. By Mueller's theorem \ref{muellertheo} this shows $\domdim R \geq n+2$.
	Now by \Thmref\ref{nTFimpliesdomdim}, $A$ also has dominant dimension at least $n$ and this finishes the proof.
\end{proof}

As a corollary of our main result, we can get the following new characterisation of the dominant dimension of a gendo-symmetric algebra:
\begin{Cor}\label{cor3dot3}
	Let $A$ be a gendo-symmetric algebra and $n \geq 1$ a natural number. Then $A$ has dominant dimension at least $n+2$ if and only it has $n$-torsion-free Auslander-Reiten sequences.
	
\end{Cor}
\begin{proof}
	Let $R:=\End_A(A \oplus D(A))$.
	By \Thmref \ref{mainresult}, $A$ has $n$-torsion-free Auslander-Reiten sequences if and only if $\domdim(A) \geq n$ and $\domdim(R) \geq n+2$. By Mueller's theorem, the dominant dimension of $R$ is equal to $\inf \{ i \geq 1 | \Ext_A^i(D(A),A) \neq 0 \}+1$ and by \Thmref\ref{fankoeresult} this is also equal to the dominant dimension of $A$. Thus $\domdim(R) \geq n+2$ alone is equivalent to $A$ having dominant dimension at least $n+2$ and this proves the corollary.
\end{proof}

The previous corollary gives an easy construction of algebras having $n$-torsion-free Auslander-Reiten sequences since one just has to find gendo-symmetric algebras with dominant dimension at least $n+2$. See for example \cite{ChMar} for a large collection of gendo-symmetric algebras from Brauer tree algebras having arbitrary large dominant dimension.
We give another example of algebras having $n$-torsion-free Auslander-Reiten sequences, using non-representation-finite hereditary algebras. For that, recall that (following \cite{CIM}) for a given finite-dimensional algebra $A$, the \emph{$m$-th SGC extension algebra} $A^{[m]}$ is defined inductively as follows:
$A^{[0]}:=A$ and for $m \geq 0$: $A^{[m+1]}:=\End_{A^{[m]}}(D(A^{[m]}) \oplus A^{[m]})$.  

\begin{example}
	Let $A=kQ$ be a path algebra with a connected quiver $Q$ that is not a Dynkin quiver. Let $A^{[m]}$ be its $m$-th SGC extension algebra. Then $A^{[m]}$ has global dimension $2m+1$ and dominant dimension $2m$ by \citep[Theorem 1.1]{CIM}.
 Thus $A^{[m]}$ satisfies $\domdim A^{[m]} \geq 2m$ and $\domdim A^{[m+1]} \geq 2m+2 $ and our main \Thmref \ref{mainresult} tells us that $A^{[m]}$ has $2m$-torsion-free Auslander-Reiten sequences.
\end{example}

In \cite{Ta}, Tachikawa showed that non-semisimple Auslander algebras cannot have reflexive Auslander-Reiten sequences. Note that non-semisimple Auslander algebras have global dimension two. We generalise this here to arbitrary Gorenstein algebras with an easy proof:
\begin{Prop}
	Let $A$ be a Gorenstein algebra of Gorenstein dimension $g$ that has $g$-torsion-free Auslander-Reiten sequences. Then $A$ is selfinjective.
\end{Prop}
\begin{proof}
	Assume that $A$ is not selfinjective and thus has positive Gorenstein dimension $g>0$. Then the module $D(A)$ has finite projective dimension equal to $g$ and by \citep[VI, Lemma 5.5]{ARS} we have $\Ext^g(D(A),A) \neq 0$.
	But by our main \Thmref \ref{mainresult}, $A$ having $g$-torsion-free Auslander-Reiten sequences implies that $\Ext^g(D(A),A)= 0$. This is a contradiction and thus $A$ has to be selfinjective.
\end{proof}

\begin{Cor}
	Let $A$ be an algebra of global dimension $g$ that has $g$-torsion-free Auslander-Reiten sequences. Then $A$ is semi-simple.
	
\end{Cor}
Recall that a module $M$ over an algebra $B$ is called \emph{$d$-cluster tilting} if $M$ is a generator-cogenerator and $\add M= \{ X \in \rmod B \mid \Ext_A^i(M,X)=0 \ \text{for} \ i=1,...,d-1\}$, see \cite{Iya}.
By Iyama's higher Auslander correspondence $d$-cluster tilting modules $M$ are in bijective correspondence with higher Auslander algebras $A$ given by sending $M$ to $A=\End_B(M)$.
\begin{Cor}\label{cor3dot7}
	Let $A \cong \End_B(M)$, where $M$ is a $d$-cluster tilting object in $\rmod B$ for a finite-dimensional algebra $B$.
	Then $B$ is $d$-representation-finite if and only if $A$ has $d$-torsion-free Auslander-Reiten sequences.
\end{Cor}
\begin{proof}
We can assume that $A$ is not semi-simple, since in this case the statement of the corollary is trivial.
Then $A$ is a higher Auslander algebra of dominant dimension $d+1$ since $M$ is $d$-cluster tilting.
	By \citep[Theorem 1.20]{Iya2}, $B$ is $d$-representation-finite if and only if $\Ext_A^i(D(A),A)=0$ for $i=1,...,d$. Thus the statement of the corollary is now a consequence of our main \Thmref \ref{mainresult} since $\Ext_A^i(D(A),A)=0$ for $i=1,...,d$ is equivalent to $\domdim \End_A(A\oplus DA)\geq d+2$ by Mueller's theorem \ref{muellertheo}.
\end{proof}

We recall the first Tachikawa conjecture, see \cite[Page 115]{Tac2}:
\begin{Conjecture}
Let $A$ be a finite dimensional algebra over a field with $\Ext_A^i(D(A),A)=0$ for all $i \geq 1$, then $A$ is selfinjective.
\end{Conjecture}
This conjecture is a consequence of the Nakayama conjecture stating that every non-selfinjective finite dimensional algebra has finite dominant dimension.
Motivated by the results of our article, we state the following conjecture:

\begin{Conjecture}
Let $A$ be a finite-dimensional algebra over a field. If $A$ has $n$-torsion-free Auslander-Reiten sequences for every natural number $n$, then $A$ is self-injective.
\end{Conjecture}

By \Thmref \ref{mainresult}, this conjecture holds if the first Tachikawa conjecture holds. In particular, this conjecture holds if the Nakayama conjecture holds.

Motivated by computer experiments with \cite{QPA}, we pose also the following question:
\begin{Question}
Let $A$ be a finite dimensional algebra with $m$ simple modules and assume that $A$ has $2m$-torsion-free Auslander-Reiten sequences. Is $A$ selfinjective?
\end{Question}
This is for example true for Nakayama algebras, since for Nakayama algebras with $m$ simple modules having dominant dimension $\geq 2m$ already implies that they are selfinjective, see \cite[Result 2]{Mar}. It would be especially interesting to know whether this question holds for general representation-finite algebras.
\section{Torsion modules and the dominant dimension of an algebra}

We now conclude by exhibiting how the dominant dimension of an algebra can be computed using torsion modules. 

For any $M\in A\m$, define $t(M)$ as the kernel of the canonical map $f_M\colon M\rightarrow M^{**}$ and denote by $I(M)$ the injective hull of $M$. Modules of the form $t(M)$ are called \emph{torsion modules}. Observe that $\Ext_A^1(\Tr M, A)$ is isomorphic to $t(M)$. Indeed for a module $M$ with minimal projective presentation $P_1 \rightarrow P_0 \rightarrow M \rightarrow 0$, the sequence $$0\rightarrow \Tr(M)^*\rightarrow P_1^{**}\rightarrow Z^*\rightarrow \Ext^1_A(\Tr M, A)\rightarrow 0$$ is exact and $P_1^{**}\rightarrow P_0^{**}$ factors through $Z^*$, where $Z$ is the kernel of the surjective map ${P_1^*\rightarrow \Tr M}$.
This means that the following diagram is commutative
\begin{equation}
	\begin{tikzcd}
		0 \arrow[r] & Z^* \arrow[r] &P_0^{**} \arrow[r] & M^{**} &\\
		& P_1^{**}\arrow[u] & & & & \\
		&	P_1 \arrow[u, "\simeq"] \arrow[r] & P_0 \arrow[uu, "\simeq"] \arrow[r]& M \arrow[r] \arrow[uu]& 0 
	\end{tikzcd}
\end{equation}
By the Snake Lemma, $t(M)\simeq \coker(P_1^{**}\rightarrow Z^*)\simeq \Ext_A^1(\Tr M, A)$.

Inspired by Proposition 2.7 of \cite{zbMATH06409569}, we get the following.

\begin{Prop}
	Let $A$ be a finite-dimensional algebra over a field. Assume that $A$ has dominant dimension at least one. Then,
	\begin{align}
		\inf\{\grade(t(M))\colon M\in A\m \}\geq \domdim A.
	\end{align}
\end{Prop}
\begin{proof}Denote by $P$ the minimal faithful projective-injective $A$-module. Let $M\in A\m$.
	By Theorem 2.8 of \cite{AB69}, $\Hom_A(\Ext_A^1(\Tr M, A), P)=\Tor_1^A(\Tr M, P)=0$. Recall that $\Ext_A^1(\Tr M, A)$ is isomorphic to $t(M)$.


Let $0\rightarrow A\rightarrow I_0\rightarrow I_1\rightarrow \cdots$ be the minimal injective resolution of $A$. By assumption, $I_i\in \add_A P$ for $0\leq i\leq n-1$, whenever $\domdim A\geq n$. Consider the exact sequences
\begin{align}
0\rightarrow \Omega^{-(i-1)}(A)\rightarrow I_{i-1}\rightarrow \Omega^{-i}(A) \rightarrow 0.
\end{align}Applying $\Hom_A(t(M), -)$ yields long exact sequences
\begin{equation}
    \begin{tikzcd}
  0 \rar & \Hom_A(t(M), \Omega^{-(i-1)}(A) ) \rar
             \ar[draw=none]{d}[name=X, anchor=center]{}
    & \Hom_A(t(M), I_{i-1})    \ar[rounded corners,
            to path={ -- ([xshift=2ex]\tikztostart.east)
                      |- (X.center) \tikztonodes
                      -| ([xshift=-2ex]\tikztotarget.west)
                      -- (\tikztotarget)}]{dll}[at end]{} \\      
\Hom_A(t(M), \Omega^{-i}(A)) \rar & \Ext_A^1(t(M), \Omega^{-(i-1)}) \rar & 0. 
\end{tikzcd} \label{eq28}
\end{equation}

Hence, the left part of these exact sequences yield $\Hom_A(t(M), \Omega^{-(i-1)}(A))=0$ for $1\leq i\leq n$. In particular, $\Hom_A(t(M), A)=0$. Now, the right part of (\ref{eq28}) gives
\begin{align}
\Ext_A^1(t(M), \Omega^{-(i-1)}(A))\simeq \Hom_A(t(M), \Omega^{-i}(A)), \quad 1\leq i\leq n.
\end{align}Therefore,
\begin{align}
\Ext_A^i(t(M), A)\simeq \Ext_A^1(t(M), \Omega^{-(i-1)}(A))=0, \quad 1\leq i\leq n-1.
\end{align}This shows that $\grade t(M)\geq n$.
\end{proof}

Actually, in \citep[Theorem 2.8]{AB69} we can replace $\Tr M$ by an arbitrary module and the degree one by any natural number, therefore, the same argument yields the following.
\begin{Cor}
Let $A$ be a finite-dimensional algebra over a field. Assume that $A$ has dominant dimension at least one. Then,
\begin{align}
\inf\{\grade(\Ext_A^i(M, A))\colon M\in \rmod A, \ i\geq 1 \}\geq \domdim A.
\end{align}
\end{Cor}

\begin{Prop}
Let $A$ be a finite-dimensional algebra over a field. Assume that $A$ has dominant dimension at least one. Then,
\begin{align*}
\inf\{\grade(t(S))\colon S \text{ simple in } A \}\leq	\inf\left\{ \grade(S)\ \middle\vert \begin{array}{c}
    S \text{ simple in } A \\
    \text{ with } I(S) \text{ not projective}
  \end{array}\right\} \leq \domdim A.
\end{align*}
\end{Prop}
\begin{proof}
Let $S$ be a simple $A$-module. Note first that since $t(S)$ is a submodule of $S$ either $t(S)=S$ or $t(S)=0$. If $t(S)=0$, then we have $\grade t(S)=+\infty$. If the injective hull of $S$ is projective, then $S$ is a first syzygy module: $S \cong \Omega^1(\Omega^{-1}(S))$ and thus $S$ is torsionless, meaning that $t(S)=0$. Here we used the equality $\TF_1=\Omega^1(\rmod A)$, that is true for general finite-dimensional algebras $A$, see for example \cite[Proposition 1.6]{AR}. This shows that 
$$ 	\inf\{\grade(t(S))\colon S \text{ simple in } A \}\leq	\inf\{\grade(S)\colon S \text{ simple in } A \text{ with } I(S) \text{ not projective}\}.$$

Assume that $\grade S\geq n\geq 1$ for every simple $A$-module such that the injective hull of $S$ is not projective. We shall proceed by induction on $n$ to show that $\domdim A\geq n$. 
Assume that $n=2$. If $\domdim A=1$, then there exists a simple module $S$ contained in the socle of $\Omega^{-1}(A)$ such that the injective hull of $S$ is not projective. We have $\grade S \geq 2 $ and thus $\Ext_A^1(S, A)=0$ and the map $\Hom_A(S, I_0)\rightarrow \Hom_A(S, \Omega^{-1}(A))$ is surjective. By assumption, $I_0$ is projective, so $S$ cannot be in the socle of $I_0$ and thus  $\Hom_A(S, \Omega^{-1}(A))$ must be zero contradicting $S$ being in the socle of $\Omega^{-1}(A)$. Therefore, such a simple module cannot exist and we obtain $\domdim A\geq 2$.

Assume now that $n\geq 3$. In particular, $\grade S\geq n-1\geq 1$ for every simple $A$-module such that the injective hull of $S$ is not projective.  By induction, $\domdim A\geq n-1$. Then, $\domdim A\geq n$ if and only if every simple module in the socle of $\Omega^{-(n-1)}$ has dominant dimension at least one. 
To obtain a contradiction, assume that there exists a simple module $S$ in the socle of $\Omega^{-(n-1)}$ so that $I(S)$ is not projective. By assumption, $\grade S\geq n$ and so $0=\Ext_A^{n-1}(S, A)\simeq \Ext_A^1(S, \Omega^{-(n-2)}(A))$. This means that the induced map $\Hom_A(S, I_{n-2})\rightarrow \Hom_A(S, \Omega^{-(n-1)}(A))$ is surjective, and so $S$ should also be in the socle of $I_{n-2}$.  This implies that $I(S)$ is a summand of $I_{n-2}$ and in particular, it implies that $I(S)$ must be a projective module contradicting the choice of $S$.
\end{proof}

Combining the previous results, we obtain the following:

\begin{Theorem}

\label{prop6dot13} 
Let $A$ be a finite-dimensional algebra over a field. Assume that $A$ has dominant dimension at least one. Then,
\begin{align*}
\domdim A &=	\inf\{\grade(\Ext_A^i(M, A))\colon M\in \rmod A, i\geq 1 \} = \inf\{\grade(t(M))\colon M\in A\m \} \\&= \inf\{\grade(t(S))\colon S \text{ simple in } A \}.
\end{align*}
\end{Theorem}

\begin{example}
In the previous theorem, the assumption on the dominant dimension of $A$ cannot be omitted.
Let $A=KQ$ be a non-semisimple path algebra of a quiver $Q$, which is not linearly oriented of Dynkin type $A_n$. Then $KQ$ has dominant dimension zero. But for every indecomposable module $M$ we have $t(M)\cong M$ if $M$ is non-projective and $t(M)=0$ if $M$ is projective. Thus we have  
$$\inf\{\grade(t(M))\colon M\in A\m \} =\inf\{\grade(t(S))\colon S \text{ simple in } A \}=1,$$
but $\domdim A=0$. \end{example}

Another sufficient condition for these equalities to hold is to require that all objects $\Ext_A^i(M, A)$ have grade greater than one for every $i\geq 1$ and all right $A$-modules $M$.

\begin{Theorem}\label{thm4dot6}
    Let $A$ be a finite-dimensional algebra over a field. If $$\inf\{\grade(\Ext_A^i(M, A))\colon M\in \rmod A, i\geq 1 \}\geq 2$$ or $\inf\{\grade(t(M))\colon M\in A\m \}\geq 2$, then 
    \begin{align*}
\domdim A &=	\inf\{\grade(\Ext_A^i(M, A))\colon M\in \rmod A, i\geq 1 \} = \inf\{\grade(t(M))\colon M\in A\m \} \\&= \inf\{\grade(t(S))\colon S \text{ simple in } A \}.
\end{align*}
\end{Theorem}
\begin{proof}
    Our aim is to prove that $\domdim A\geq 1$, under the assumption that \begin{equation}
        \inf\{\grade(\Ext_A^1(M, A))\colon M\in \rmod A \}\geq 2. \label{eq16}
    \end{equation} For this, it is enough to show that the injective hull of $A$, $I_0(A)$, is projective. By Corollary 2.5 of \cite{Rei}, it is then enough to show that $\Hom_A(\Ext_A^1(M, A), I_0(A))=0$ for all $M\in \rmod A$.

   \textbf{Claim 1.} Given $M\in A\m$, if $\grade M>0$, then $\grade M\geq 2$. 
   
To show Claim 1, let $M\in A\m$ have positive grade. Then $M^*=\Hom_A(M, A)=0.$ Consider $P_1\rightarrow P_0\rightarrow M\rightarrow 0$. Applying $\Hom_A(-, A)$ yields the exact sequence $$0\rightarrow M^*=0\rightarrow P_0^*\rightarrow P_1^*\rightarrow \Tr M\rightarrow 0.$$ Applying $(-)^*$ once more, we obtain  the exact sequence
$$0\rightarrow (\Tr M)^*\rightarrow P_1^{**}\rightarrow P_0^{**}\rightarrow \Ext_A^1(\Tr M, A) \rightarrow \Ext_A^1(P_1^*, A)=0.$$ Thus, $M\simeq \Ext_A^1(\Tr M, A)$ as right $A$-modules. By assumption, $\grade \Ext_A^1(\Tr M, A)\geq 2$, and so it follows that $M$ has grade at least two, proving Claim 1.

Let $M\in \rmod A$. Let $X$ be an $A$-submodule of $\Ext_A^1(M, A).$ Define $\Tilde{X}:=\Ext_A^1(M, A)/X$. 
Applying $(-)^*$ to the exact sequence $
    0\rightarrow X\rightarrow \Ext_A^1(M, A) \rightarrow \Tilde{X}\rightarrow 0$ provides the exact sequence \begin{align*}
        0\rightarrow \Tilde{X}^*\rightarrow \Ext_A^1(M, A)^*\rightarrow X^*\rightarrow \Ext_A^1(\Tilde{X}, A)\rightarrow \Ext_A^1(\Ext_A^1(M, A), A).
    \end{align*}
    Since $\grade \Ext_A^1(M, A)\geq 2$ this exact sequence collapses to $\Tilde{X}^*=0$ and $X^*\simeq \Ext_A^1(\Tilde{X}, A)$. Hence, $\grade \Tilde{X}\geq 1$. By Claim 1, $\grade \Tilde{X}\geq 2$, and therefore $X^*\simeq \Ext_A^1(\Tilde{X}, A)=0.$ Thus, $\Hom_A(X, \soc A)=0$. Since $X$ is arbitrary, it follows that $\Hom_A(\Ext_A^1(M, A), I(\soc A))=0$ for all $M\in \rmod A$. So, $\domdim A\geq 1$. So, assuming that (\ref{eq16}) holds, the equalities now follow from \Thmref \ref{prop6dot13}.

    Now, it remains to see that both inequalities in the assumption of the theorem imply (\ref{eq16}).

Observe that, if $M$ is an indecomposable projective module, then $\grade (\Ext_A^1(M, A))=+\infty$. Otherwise, if $M$ is an indecomposable non-projective module, then there exists an indecomposable non-projective $M'\in A\m$ so that $\Tr M'\simeq M$ (see \citep[IV, Proposition 1.7]{ARS}. Hence, $\Ext_A^1(M, A)\simeq t(M')$ and  $$\inf\{\grade(\Ext_A^1(M, A))\colon M\in \rmod A \}\geq \inf\{\grade(t(M)\colon M\in A\m \}.$$
    Since $\inf\{\grade(\Ext_A^1(M, A))\colon M\in \rmod A \}\geq \inf\{\grade(\Ext_A^i(M, A))\colon M\in \rmod A, i\geq 1 \}$, the result follows.
\end{proof}

We note, however, that it is not enough to consider the torsion modules  of simple modules in the assumption of the previous theorem. Indeed, due to Ringel, there exists a finite-dimensional algebra $A$ with two simple modules, both being torsionless, and with $\domdim A=0$ (see \citep[2.1]{zbMATH07412003}. ) So for such an algebra $A$ we have $\inf\{\grade(t(S))\colon S \text{ simple in } A \}=+\infty$.

We finish with the following application of \Thmref\ref{prop6dot13}.

\begin{Cor}\label{cor4dot7}
Let $A$ be a higher Auslander algebra of global dimension $g$.
Then every non-zero torsion $A$-module has projective dimension equal to $g$.
\end{Cor}
\begin{proof}
Let $N=t(M)$ be a non-zero torsion module.
Then we have $\grade(N)\geq \domdim A=g$ by \ref{prop6dot13}. But we also have $g\leq\grade(N) \leq \pdim (N) \leq \gldim A=g$.
Thus $\pdim(N)=g$ for all non-zero torsion modules $N$.
\end{proof}

\section*{Acknowledgements}
This work began when the first-named author was visiting the Max Planck Institute for Mathematics at Bonn, and so the first-named author would like to thank the MPIM for its hospitality during the stay.



\bibliographystyle{alphaurl}
\bibliography{bibarticle}


\end{document}